\documentclass[12pt]{article}
\usepackage{amscd,amsfonts,amsmath,amssymb,amsthm,amstext,latexsym}
\usepackage{enumerate,pifont,epsf,graphicx}
\usepackage[english]{babel}
\usepackage[T1]{fontenc}
\usepackage{epsfig}
\usepackage[latin1]{inputenc}

\def\R{\mathbb{R}}

\def\C{\mathbb{C}}
\def\H{\mathbb{H}}
\def\M{\mathbb{M}}

\def\r{\mathbb R}
\def\h{\mathbb H}

\def\0{\mathbf 0}

\newtheorem{theorem}{Theorem}
\newtheorem{lemma}[theorem]{Lemma}

\newtheorem{remark}[theorem]{Remark}

\newtheorem{assertion}[theorem]{Assertion}

\begin{document}

\begin{title}
  {Simply-connected minimal surfaces with finite total curvature in
    $\H^2\times\R$}
\end{title}
\vskip .2in

\begin{author}
  {Juncheol Pyo\thanks{Research partially supported by the CEI BioTIC
      GENIL project (CEB09-0010) and the Basic Science Research Program through the National Research Foundation of Korea (NRF) funded by the Ministry of Education, Science and Technology (2012-0007728).}  and M.
    Magdalena Rodr\'\i guez\thanks{Research partially supported by the
      MEC-FEDER Grant no. MTM2011-22547 and the Regional
      J. Andaluc\'\i a Grant no. P09-FQM-5088.} }
\end{author}
\maketitle

\begin{abstract}
  Laurent Hauswirth and Harold Rosenberg developed in~\cite{hr} the
  theory of minimal surfaces with finite total curvature in
  $\H^2\times\R$. They showed that the total curvature of one such a
  surface must be a non-negative integer multiple of $-2\pi$. The
  first examples appearing in this context are vertical geodesic
  planes and Scherk minimal graphs over ideal polygonal domains. Other
  non simply-connected examples have been constructed recently
  in~\cite{mmr, mr,p}.

  In the present paper, we show that the only complete minimal
  surfaces in $\H^2\times\R$ of total curvature $-2\pi$ are Scherk
  minimal graphs over ideal quadrilaterals. We also construct properly
  embedded simply-connected minimal surfaces with total curvature
  $-4k\pi$, for any integer $k\geq 1$, which are not Scherk
  minimal graphs over ideal polygonal domains.
\end{abstract}

\noindent{\it Mathematics Subject Classification:} Primary 53A10,
Secondary 49Q05, 53C42

\section{Introduction}

In the classical theory of minimal surfaces in $\R^3$, the ones better
known are those with finite total curvature. We recall that the total
curvature of a surface $M$ is defined as $C(M)=\int_M K$, where $K$
denotes the Gauss curvature of $M$. If a minimal surface $M$ of $\R^3$
has finite total curvature (i.e.  $|C(M)|<+\infty$) then either $M$ is
a plane or it must be $C(M)=-4\pi k$, for some integer $k\geq 1$, and the
equality only holds for $M$ being the catenoid or Enneper's surface
(see~\cite[Theorems 9.2 and 9.4]{oss}).

In the last decade, the geometry of minimal surfaces in $\H^2\times\R$
has been actively studied, and many examples have been constructed
(see for instance~\cite{cr, h, mrr, mor, nr, et}).
Hauswirth and Rosenberg started in~\cite{hr} the study of complete
minimal surfaces of finite total curvature in $\Bbb H^2\times \Bbb
R$. The only known examples at that moment were the Scherk minimal
graphs over ideal polygonal domains with an even number of edges, with
boundary values $\pm\infty$ disposed alternately. Morabito and the
authors constructed in~\cite{mr, p} non simply-connected properly
embedded minimal surfaces with finite total curvature and genus
zero. Quite recently, in a joint work with Mart\'\i n and Mazzeo, the
second author~\cite{mmr} has constructed properly embedded minimal surfaces with
finite total curvature and positive genus.

The classification of minimal surfaces of finite total curvature in
$\H^2\times\R$ arises very naturally. The first result of
classification appearing in this theory was that the only complete
minimal surfaces with vanishing total curvature are the vertical
geodesic planes (see \cite[Corollary 5]{hst}). Quite recently,
Hauswirth, Nelli, Sa Earp and Toubiana have proved in~\cite{hnst} that
a complete minimal surface in $\h^2\times\r$ with finite total
curvature and two ends, each one asymptotic to a vertical geodesic
plane, must be one of the horizontal catenoids constructed
in~\cite{mr, p}.  In this paper, we show that the Scherk minimal
graphs over ideal quadrilaterals (i.e. ideal polygonal domains bounded
by four ideal geodesics) are the only complete minimal surfaces of
total curvature $-2\pi$.

It was expected that each end of a minimal surface with finite
curvature in $\h^2\times\r$ were asymptotic to either a vertical
geodesic plane or a Scherk graph over an ideal polygonal domain. We
construct new simply-connected examples, that we call {\it twisted
  Scherk examples}, that highlight this is not the case. They all have
total curvature an integer multiple of $-4\pi$, so we cannot expect a
classification result for Scherk graphs over ideal polygonal domains
bounded by $4k+2$ edges as the only simply-connected complete minimal
surfaces in $\H^2\times\R$ with total curvature $-4k\pi$.

\section{Preliminaires}
\label{sec:prelim}

We consider the Poincar\'e disk model of $\H^2$; i.e.  $\H^2=\{z\in\C\
|\ |z|<1\}$, with the hyperbolic metric
$g_{-1}=\frac{4}{(1-|z|^2)^2}|dz|^2$. We denote by
$\partial_\infty\H^2$ the infinite boundary of $\H^2$
(i.e. $\partial_\infty\H^2=\{z\in\C\ |\ |z|=1\}$) and by $\0$ the
origin of $\H^2$. Also $t$ will denote the coordinate in $\R$.

Let $M$ be a complete orientable minimal surface immersed in
$\H^2\times\R$.  We define the total curvature of $M$ as $C(M)=\int_M
K$, where $K\leq 0$ denotes the Gaussian curvature of $M$.
We say that $M$ has finite total curvature when $|C(M)|<+\infty$.

In this section we summarize the geometric properties of minimal
surfaces in $\H^2\times\R$ with finite total curvature given by
Hauswirth and Rosenberg in~\cite{hr}.

We call the {\it height function} of $M$ the horizontal projection $h:M\to
\R$, and we denote by $F$ the vertical projection of $M$ over
$\H^2$. It is well-known that $h$ is a real harmonic function on $M$
and that $F$ is an harmonic map from $M$ to $\H^2$.  Given a conformal
parameter $w$ on $M$, Sa Earp and Toubiana~\cite{et} proved that
$(h_w)^2=-Q$, where $Q$ is the Hopf differential associated to
$F$. Then the zeroes of $Q$ are of even order and, up to a sign (which
corresponds to a reflection symmetry with respect to $\H^2\times\R$),
\begin{equation}\label{eq:h}
  h=\Re\left(-2i\int\sqrt{Q}\right),
\end{equation}
see equation (3) in~\cite{hr}.

We fix a unit normal vector field $N$ on $M$.  We now state the main
theorem in~\cite{hr}.

\begin{theorem}\cite{hr}\label{th:description}
  Let $M$ be a complete, orientable, minimal surface immersed in
  $\H^2\times\R$ with finite total curvature. Then:
  \begin{enumerate}
  \item $M$ is conformally a closed Riemann surface $\M$ punctured in a
    finite number of points $p_1,\cdots,p_n$, called {\em ends} of
    $\Sigma$.
  \item $Q$ is holomorphic on $M$ and extends meromorphically to its
    ends $p_i$. If we parameterize conformally a neighborhood of $p_i$
    in $M$ by $\Omega=\C\setminus D_0$, where $D_0$ is the open unit
    disk in~$\C$ centered at the origin, then
    \[
    Q(z)= z^{2m_i} (dz)^2,
    \]
    for some integer $m_i\geq -1$.
  \item $N_3=\langle N,\partial_t\rangle$ converges uniformly to zero
    on each end $p_i$.
  \item The total curvature of $M$ is given by
    \begin{equation}\label{eq}
      \int_M K=2\pi\left(2-2g-2n-\sum_{i=1}^n m_i\right).
    \end{equation}
  \end{enumerate}
\end{theorem}

\begin{remark}\label{rem:Q}
Suppose $p_i$ is an end of $M$ for which $m_i=-1$. If we want to
close periods in equation~\eqref{eq:h}, then we have to choose
$Q(z)= -z^{-2} (dz)^2$, $z\in\Omega$.
\end{remark}

\begin{assertion}\label{lemma}
  In the second item of Theorem~\ref{th:description}, $m_i$ cannot
  equal $-1$.
\end{assertion}

\begin{proof}
  Suppose $M$ (in the setting of Theorem~\ref{th:description}) has an
  end $p_1$ for which $m_1=-1$. We know that a neighborhood
  $\mathcal{E}$ of $p_1$ can be conformally parameterized on
  $\Omega=\{z\in\C\ |\ |z|\geq 1\}$, where $Q(z)=-z^{-2}dz^2$ (see
  Remark~\ref{rem:Q}). From~\eqref{eq:h} we then get $h(z)=2
  \Re\left(\int_M \frac{dz}z\right)=2 \ln|z|$. Therefore,
  $\mathcal{E}$ is a vertical annulus whose intersection with each
  horizontal slice $\H^2\times\{t\}$, $t\geq 0$, is a compact curve.

  The boundary of $\mathcal{E}$ (which corresponds to $\{|z|=1\}$)
  consists of a horizontal compact curve~$\Gamma$ at height
  zero. Consider $R>0$ big enough so that the disc $D\subset\H^2$ of
  radius $R$ centered at the origin contains $\Gamma$ in its interior.
  And let $\mathcal{C}$ be the complete vertical rotational catenoid
  constructed by Nelli and Rosenberg in~\cite{nr} whose neck is
  $\partial D$. Since $\mathcal{E}$ intersects each horizontal slice
  in a compact curve, we deduce using the Maximum Principle with
  vertically translated copies of $\mathcal{C}$ that $\mathcal{E}$
  must be contained in $D\times\R$. But this is not possible: If we
  translate $\mathcal{C}$ vertically up a distance $\pi$, we reach a
  contradiction by applying the Maximum Principle with the family of
  shrunk catenoids going from $\mathcal{C}$ to the 2-sheeted covering
  of the punctured slice $(\H^2-\{\0\})\times\{\pi\}$.
\end{proof}

We finish this section by describing the asymptotic behavior of a
complete, orientable, minimal surface immersed in $\H^2\times\R$ with
finite total curvature.

\begin{lemma}\cite{hr}\label{lem:description}
  Let $M$ be a minimal surface in the hypothesis of
  Theorem~\ref{th:description}, and $p_i$ an end of~$M$. If $m_i\geq
  0$ is the integer associated to $p_i$, as defined in
  Theorem~\ref{th:description}, then $p_i$ corresponds to $m_i+1$
  geodesics
  $\gamma_1,\ldots,\gamma_{m_i+1}\subset\H^2\times\{+\infty\}$,
  $m_i+1$ geodesics
  $\Gamma_1,\ldots,\Gamma_{m_i+1}\subset\H^2\times\{-\infty\}$, and
  $2(m_i+1)$ vertical straight lines (possibly some of them coincide) in $\partial_\infty\H^2\times\R$,
  each one joining an endpoint of some $\gamma_j$ to an endpoint of
  some $\Gamma_j$.
%
\end{lemma}

\section{Minimal examples with finite total curvature}
\label{sec:examples}

Given any two points $p,q\in\H^2\cup\partial_\infty\H^2$, we will
denote by $\overline{pq}$ the geodesic arc joining $p,q$.

We consider an even number of different points
$p_1,\cdots,p_{2k}\in\partial_\infty\H^2$ (cyclically ordered), with
$k\geq 2$, and we call $A_i=\overline{p_{2i-1} p_{2i}}$,
$B_i=\overline{p_{2i} p_{2i+1}}$, for any $1\leq i\leq k$, where we
consider the cyclic notation $p_{2k+1}=p_1$. Let $\Omega$ be the ideal
polygonal domain bounded by $A_1,B_1,\cdots,A_k,B_k$. We call {\it
  Scherk minimal graph} over $\Omega$ to a minimal graph over $\Omega$
with boundary values $+\infty$ over the $A_i$ edges and $-\infty$ over
the $B_i$ edges (in~\cite{cr, nr} it is proved that it exists and it
is unique up to a vertical translation). In~\cite{cr,hr} it is proved
that such a graph has total curvature $2\pi(1-k)$.  Scherk graphs over
ideal polygonal domains, together with the vertical geodesic planes,
where the first known examples of minimal surfaces with finite total
curvature.

In~\cite{mr, p} other non-simply-connected examples where presented,
called {\it minimal $k$-noids}. We briefly explain their construction:
Consider an even number of points $p_1,\cdots,p_{2k}$ (cyclically
ordered) such that $p_{2i-1}\in\H^2$ and
$p_{2i}\in\partial_\infty\H^2$. We call $A_i=\overline{p_{2i-1}
  p_{2i}}$ and $B_i=\overline{p_{2i} p_{2i+1}}$. Consider the minimal
graph $\Sigma$ over the polygonal domain bounded by
$A_1,B_1,\cdots,A_k,B_k$ with boundary values $+\infty$ over the $A_i$
edges and $-\infty$ over the $B_i$ edges (it exists and is unique up
to a vertical translation, by~\cite{cr,mrr}), which has total
curvature $2\pi(1-k)$ (see~\cite{cr}). The conjugate minimal surface
$\Sigma^*$ of $\Sigma$ is a minimal graph contained in
$\H^2\times\{t\geq 0\}$, whose boundary consists of $k$ geodesic
curvature lines in $\H^2\times\{0\}$. (The conjugation for minimal
surfaces in $\H^2\times\R$ was introduced by Daniel~\cite{d} and by
Hauswirth, Sa~Earp and Toubiana~\cite{hst}.) If we reflect $\Sigma^*$
with respect to $\H^2\times\{0\}$, we get a properly embedded minimal
surface of genus zero, $k$ ends asymptotic to vertical geodesic planes
and total curvature $4\pi(1-k)$. For $k=2$, the obtained examples are
usually called {\em horizontal catenoids}, and have been recently
classified by Hauswirth, Nelli, Sa Earp and Toubiana as the only
complete minimal surfaces in $\h^2\times\r$ with finite total
curvature and two ends, each one asymptotic to a vertical geodesic
plane.

Using a gluing method, the second author has recently constructed in a
joint work with Mart\'\i n and Mazzeo a wide range of properly
embedded minimal surfaces with finite total curvature and finite
topology (with possibly positive genus).

We wondered if Scherk minimal graphs were, together with the vertical
geodesic planes, the only complete, embedded, simply-connected
examples of finite total curvature. In this section we explain the
simple construction of other different complete, embedded,
simply-connected examples, that we will call {\it twisted Scherk examples}.

\subsection{
  Twisted Scherk examples}

Let us first construct an example with total curvature $-4\pi$. Let
$p_1,p_2$ be two points in $\partial_\infty\H^2$. Up to an isometry of
$\H^2$, we can assume $p_{1}=1$ and $p_{2}=e^{i\theta}$, for some
fixed $\theta\in(0,\pi/2]$ (see Figure~\ref{k1}). We call
$A_1=\overline{\0 p_{1}}$, $B_{1}=\overline{p_{1} p_{2}}$ and
$C_{1}=\overline{\0 p_{2}}$. Let $\Delta$ be the geodesic triangle
bounded by $A_{1}\cup B_{1}\cup C_{1}$. By the triangle inequality at
infinity (see~\cite[Lemma 3]{cr}), we get that $\Delta$ satisfies the
Jenkins-Serrin condition for the existence of a minimal graph~$u$ over
$\Delta$ with boundary values $+\infty$ on $A_{1}$, $-\infty$ on
$B_{1}$ and $0$ on $C_{1}$ (see~\cite[Theorem 3]{cr} and~\cite[Theorem
3.3]{mrr}).

Now let us see that the graph surface $\Sigma(u)$ of $u$ has finite
total curvature: For any positive integer $n$, we denote $r=1-1/(n+1)$
and $p_{1,n}=r$, $p_{2,n}=re^{i\theta}$. By Theorem 3 in~\cite{nr},
there exists a minimal graph $u_{r}(n)$ over the geodesic triangle of
vertices $\0,p_{1,n},p_{2,n}$ taking boundary values $+n$ on
$\overline{\0 p_{1,n}}$, $-n$ on $\overline{p_{1,n} p_{2,n}}$ and $0$
on $\overline{\0 p_{2,n}}$. By the Gauss-Bonnet formula, the graph
surface of $u_{r}(n)$ has total curvature $\pi$.  Since $u_{r}(n)$
converges uniformly on compact sets of $\Delta$ to $u$ as
$n\rightarrow\infty$, the total curvature of $\Sigma(u)$ is at most
$\pi$, and then finite.

By rotating $\Sigma(u)$ an angle $\pi$ about the horizontal geodesic
ray $\overline{\0p_{2}}$ contained in its boundary, we obtain a
minimal graph whose boundary consists of the vertical geodesic
$\{\0\}\times\R$. We extend such a graph by rotation of angle $\pi$
about its boundary, and we get a properly embedded simply-connected
minimal surface $\Sigma_{1}$.

\begin{figure}
\begin{center}
\includegraphics[height=5cm, width=11cm]{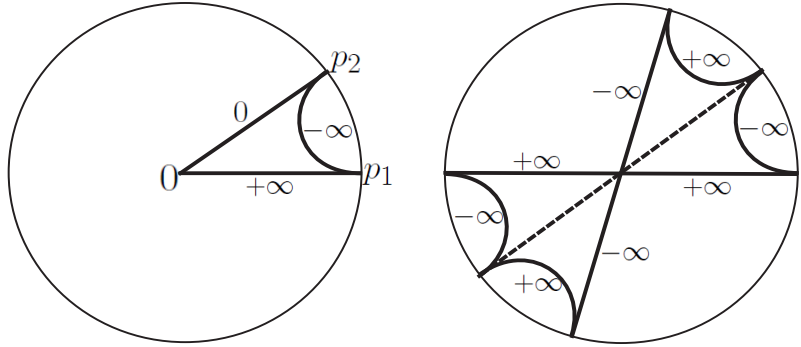}
\end{center}
\caption{Left: The minimal graph over the triangle region
  with these prescrived values is the fundamental piece of a twisted
  Scherk example $\Sigma_1$ with total curvature $-4\pi$. Right:
  Vertical projection of $\Sigma_1$.  }
\label{k1}
\end{figure}

Since $\Sigma_1$ consists of four copies of $\Sigma(u)$, then it has
finite total curvature. Then equation~\eqref{eq} applies. In our case,
$g=0$, $n=1$ and $m_1=2$ ($m_1=2$ follows from the fact that the
intersection of $M$ with a horizontal slice $\H^2\times\{t\}$, for
$t>0$ large enough, consists of three divergent curves, see
Figure~\ref{k1}). Thus $\int_{\Sigma_1}K=-4\pi$.

\medskip

Now, let us consider $k\geq 2$. Let $\Omega$ be a polygonal domain
whose vertices are $\0$ and $2k-1$ different ideal points
$p_1,\cdots,p_{2k-1}\in\partial_\infty\H^2$. Assume that $\Omega$
satisfies the Jenkins-Serrin condition of Theorem 3 in~\cite{cr} or
Theorem 3.3 in~\cite{mrr}. The example below proves that there exist
such domains. We call $\Sigma$ the minimal graph over $\Omega$ with
boundary values $+\infty$ on $\overline{\0 p_1}$ and on
$\overline{p_{2i} p_{2i+1}}$, for $1\leq i\leq k-1$; and $-\infty$ on
$\overline{p_{2i-1} p_{2i}}$, for $1\leq i\leq k-1$, and zero on
$\overline{\0 p_{2k-1}}$. By rotating $\Sigma$ an angle $\pi$ about
the vertical geodesic line $\{\0\}\times\R$ in its boundary, we obtain
a properly embedded simply-connected minimal surface
$\Sigma_{k}$. Arguing similarly as for $\Sigma_1$, we can prove that
$\int_{\Sigma_k}K=-4k\pi$. Then we have proved the following theorem.

\begin{theorem}
  For any integer $k\geq1$, there exists a properly embedded
  simply-connected minimal surface $\Sigma_k$ of finite total
  curvature $-4k\pi$ which is not a minimal (vertical) graph.
\end{theorem}

Now let us construct a polygonal domain $\Omega$ in the above
setting. For any $\theta\in(0,\frac{\pi}{2k})$, let $\Omega_\theta$ be
the polygonal domain with vertices $\0$, $\widetilde p_{1}=1$, and
\[
p_{n}=e^{i(n-1)\theta},\quad 2\leq n\leq k+1.
\]
We mark by $+\infty$ the edge $\overline{\0, \widetilde p_1}$ and
those of the form $\overline{p_{2i} p_{2i+1}}$; by $-\infty$ the edges
of the form $\overline{p_{2i-1} p_{2i}}$; and by $0$ the edge
$\overline{\0 p_{k+1}}$. It is clear that $\Omega_\theta$ does not
satisfy the Jenkins-Serrin condition (see Theorem 3 in~\cite{cr} or
Theorem 3.3 in~\cite{mrr}), as we can consider the inscribed polygonal
domain $\mathcal{P}\subset\Omega$ with vertices $\0,\widetilde
p_1,p_2,p_3$ and any choice of disjoints horocycles $H_1,H_2,H_3$ at
$\widetilde p_1,p_2,p_3$ respectively, for which
$
\mbox{dist}_{\H^2}(\0,H_1)+ \mbox{dist}_{\H^2}(H_2,H_3)=
\mbox{dist}_{\H^2}(\0,H_3)+
\mbox{dist}_{\H^2}(H_1,H_2)
$.

To solve this problem, we consider a small perturbation of $\widetilde
p_1$: Let $\Omega_{\theta,\beta}$ be the polygonal domain with
vertices $p_1=e^{-i\beta}$, for $\beta\in(0,\frac{\pi}2-k\theta]$
small, and $p_n$ defined as above, for $2\leq n\leq k+1$. This domain
$\Omega_{\theta,\beta}$ satisfies the Jenkins-Serrin condition if we
label by $+\infty$ the edge $\overline{\0, p_1}$ and those of the form
$\overline{p_{2i} p_{2i+1}}$; by $-\infty$ the edges of the form
$\overline{p_{2i-1} p_{2i}}$; and by $0$ the edge $\overline{\0
  p_{k+1}}$.

Let $R$ be the reflection with respect to the geodesic containing
$\overline{\0 p_{k+1}}$. Then $\Omega= \Omega_{\theta,\beta}\cup
R(\Omega_{\theta,\beta})$ is in the desired conditions. See
Figure~\ref{k2}. 

\begin{figure}
\begin{center}
\includegraphics[height=5cm, width=11cm]{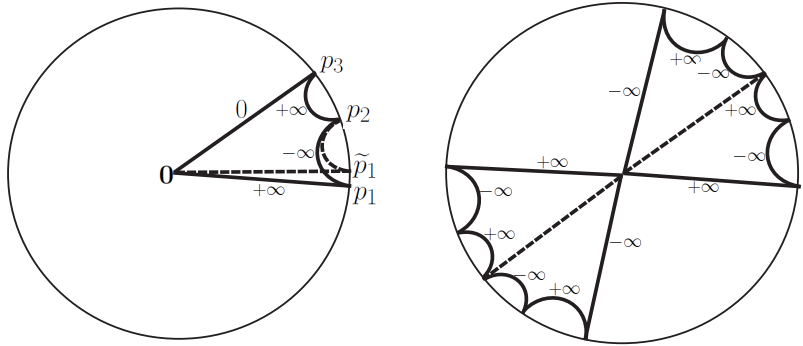}
\end{center}
\caption{Left: The fundamental piece of a twisted
  Scherk example $\Sigma_2$ with total curvature $-8\pi$. Right:
  Vertical projection of $\Sigma_2$.  }
\label{k2}
\end{figure}

%
%
%
%

\section{Uniqueness of Scherk minimal graphs}

\begin{theorem}
  If $M$ is a complete minimal surface of total curvature $-2\pi$ in
  $\H^2\times\R$, then $M$ is the Scherk minimal graph over an ideal
  quadrilateral.
\end{theorem}
\begin{proof}
  Since the total curvature of $M$ is $-2\pi$, we have by equation
  (\ref{eq}) in Theorem \ref{th:description} that
  \[
  -2\pi=2\pi\left(2-2g-2n-\sum_{i=1}^n m_i\right).
  \]
  We already know that $m_i \geq 0$, by Assertion~\ref{lemma}. And
  $n\geq 1$, since a complete minimal surface in $\H^2\times\R$ cannot
  be compact.  So the only possibility is $g=0$, $n=1$ (hence the
  complete minimal surface $M$ is simply-connected) and $m_1=1$.

  As $m_1=1$, we know by Lemma~\ref{lem:description} that there are
  four points $p_{1},p_2,p_3,p_{4}\in\partial_{\infty}\H^2$, with
  $p_i\neq p_{i+1}$ for any $i$, such that the end of $M$ corresponds
  to
  \[
  (\overline{p_{1}p_{2}}\times\{+\infty\})
  \cup(\overline{p_{2}p_{3}}\times\{-\infty\})
  \cup(\overline{p_{3}p_{4}}\times\{+\infty\})
  \cup(\overline{p_{4}p_{1}}\times\{-\infty\}),
  \]
  together with the complete vertical  geodesics $\{p_i\}\times\R$ in
  the ideal cylinder $\partial_\infty\H^2\times\R$ joining their
  endpoints.

  Let us now prove that the four points $p_i$ are all different. By
  the maximum principle using vertical geodesic planes, we know that
  at least three of them are different as $M$ cannot be a vertical
  plane. Suppose $p_1=p_3$ (the case $p_2=p_4$ follows
  similarly). Also using the maximum principle with vertical geodesic
  planes, we get that the vertical projection $\pi(M)$ of $M$ is
  contained in the ideal geodesic triangle of vertices
  $p_1,p_2,p_4$. Even more, $\pi(M)$ is contained in a domain ${\cal
    T}\subset\H^2$ bounded by $\overline{p_{1}p_{2}}$,
  $\overline{p_{1}p_{4}}$ and a strictly concave (with respect to
  ${\cal T}$) curve $\alpha$. We observe that the points in $M$
  projecting onto $\alpha$ have horizontal normal vector.  Suppose
  that the vertical projection of the limit normal vector of $M$ (that
  we also call $N$) along $\overline{p_{1}p_{2}}\times\{+\infty\}$
  points to ${\cal T}$. We observe that the horizontal curves in $M$
  with endpoint in $\{p_2\}\times\R$ arrive orthogonally to
  $\partial_\infty\H^2\times\R$. In particular, $N$ is constant along
  the vertical asymptotic line $\{p_2\}\times\R$. On one hand that
  implies, looking at the behavior of $N$ along the asymptotic
  boundary of $M$ (corresponding to the end) that the vertical
  projection of $N$ along $\overline{p_{1}p_{2}}\times\{-\infty\}$
  also points to ${\cal T}$, and its projection along
  $\overline{p_{1}p_{4}}\times\{\pm\infty\}$ goes out from ${\cal
    T}$. On the other hand, if we follow the projection of $N$ along
  $\alpha$, we obtain that it points to ${\cal T}$ along
  $\overline{p_{1}p_{4}}\times\{\pm\infty\}$, a contradiction.

  We now claim that $p_1,p_2,p_3,p_4$ are cyclically ordered. We
  define the solid cylinder $C_{r,T}=\{(z,t): |z|\leq r, |t|\leq T\}$,
  for $r<1$ close to one and $T$ large, and consider $M_{r,T}=M\cap
  C_{r,T}$, which is a compact minimal surface bounded by two
  horizontal compact curves 
  contained in $\{t=T\}$ close to $\overline{p_{1}p_2}\times\{T\}$ and
  $\overline{p_3p_4}\times\{T\}$, two curves 
  on $\{t=-T\}$ close to $\overline{p_2p_{3}}\times\{-T\}$ and
  $\overline{p_4p_1}\times\{-T\}$, and four curves on $\{|z|=r\}$
  close to vertical lines. By the flux formula with respect to the
  Killing vector field $\partial_t$ (see~\cite[Proposition 3]{hlr}),
  we have
  \begin{equation}\label{equation}
    \int_{\partial M_{r,T}}\langle \nu,\partial_t\rangle=0,
  \end{equation}
  where $\nu$ is the outward-pointing unit conormal to $M_{r,T}$ along
  $\partial M_{r,T}$. We get from~\eqref{equation}, taking limits as
  $r\rightarrow 1$ and $T\to +\infty$, that
  $|\overline{p_{1}p_2}|+|\overline{p_{3}p_4}|=|\overline{p_{2}p_3}|+|\overline{p_4p_1}|$,
  where $|\bullet|$ denotes (as in~\cite{cr}) the hyperbolic length of
  the curve $\bullet$ outside some disjoint horocycles at the ideal
  points $p_i$, identifying $\H^2$ with the corresponding horizontal
  slice. By the triangle inequality at infinity~\cite[Lemma 3]{cr} we
  get that $p_1,p_2,p_3,p_4$ must be cyclically ordered.

  We call $\Omega$ the ideal quadrilateral with vertices
  $p_{1},p_2,p_3,p_{4}$.  By the maximum principle using vertical
  geodesic planes, we get that $\pi({M})\subset\Omega$.  On the other
  hand, the geometry of the end of $M$ says that a neighborhood of
  $\partial\Omega$ is contained in $\pi(M)$. Since $M$ is complete and
  simply-connected, we conclude $\pi({M})=\Omega$.

  Now let us show that the normal vector of $M$ is never
  horizontal. Suppose there exists a point $P\in M$ such that
  $N_{3}(P)=0$. Let $\Gamma\times\R$ be the vertical geodesic plane
  tangent to $M$ at $P$. Since $M$ and $\Gamma\times\R$ have first
  contact order at $P$, their intersection consists of $k$ curves
  meeting at equals angles at $P$, with $k\geq 2$. Thus, there are at
  least four branches of $M\cap(\Gamma\times\R)$ leaving $P$ (see
  Figure~\ref{intersection}, left). Since $M$ is simply-connected, we
  deduce using the maximum principle with vertical planes that there
  cannot exists a compact cycle in $M\cap(\Gamma\times\R)$. Hence
  $\Gamma$ cannot intersect two edges of $\Omega$, so it must have
  some $p_i$ as an endpoint.  Denote by $\gamma=\gamma(t)$, $t\in\R$,
  the arc-length parameterized geodesic of $\H^2$ orthogonal to
  $\Gamma$ such that $\gamma(0)=\pi(P)$; and by $\Gamma_t$ the
  geodesic of $\H^2$ passing through $\gamma(t)$ orthogonally (in
  particular, $\Gamma_0=\Gamma$). For $\varepsilon>0$ small,
  $\Gamma_\varepsilon$ intersects two edges of $\Omega$, say
  $\overline{p_{1}p_{2}}$ and $\overline{p_{2}p_{3}}$, and the number
  of intersection curves between the vertical plane
  $\Gamma_\varepsilon\times\R$ and $M$ is at least two (see
  Figure~\ref{intersection}, right). But only one branch of the
  intersection curves can arrive to
  $\overline{p_{1}p_{2}}\times\{+\infty\}$ (resp.
  $\overline{p_{2}p_{3}}\times\{-\infty\}$), the other branch should be a compact loop, a contradiction.

  \begin{figure}\label{intersection}
    \begin{center}
      \includegraphics[height=5cm, width=6.5cm]{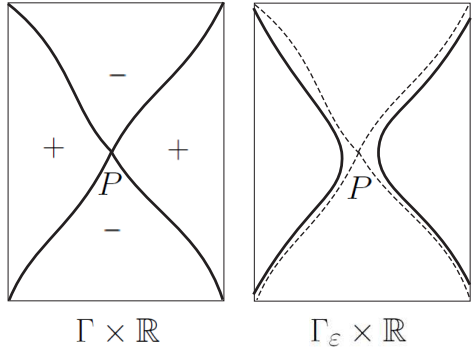}
    \end{center}
    \caption{Left: The nodal domains between $M$ and $\Gamma\times\R$
      at a point with horizontal normal vector.  Right: The
      intersection curves between $M$ and
      $\Gamma_\varepsilon\times\R$.}
    \label{intersection}
  \end{figure}


  We have prove then that, for any point $q\in\Omega$, the
  intersection of $\{q\}\times\R$ with $M$ is transverse. So the
  number of intersection points does not depend on $q$.  For $q$ near
  an edge of $\Omega$ this number is one. We conclude that $M$ is a
  graph over $\Omega$.
\end{proof}

\bibliographystyle{plain}

\mbox{}\\

\noindent
Juncheol Pyo\\
Department of Mathematics\\
 Pusan National University\\
  Busan 609-735, Korea\\
e-mail: \texttt{jcpyo@pusan.ac.kr}\\

\noindent
M. Magdalena Rodr\'\i guez\\
Departamento de Geometr\'\i a y Topolog\'\i a\\
Universidad de Granada\\
Fuentenueva, 18071, Granada, Spain\\
e-mail: \texttt{magdarp@ugr.es}

\end{document}